\documentclass[pdflatex,sn-mathphys-num]{sn-jnl}

\usepackage{graphicx}%
\usepackage{multirow}%
\usepackage{amsmath,amssymb,amsfonts}%
\usepackage{amsthm}%
\usepackage{mathrsfs}%
\usepackage[title]{appendix}%
\usepackage{xcolor}%
\usepackage{textcomp}%
\usepackage{manyfoot}%
\usepackage{booktabs}%
\usepackage{algorithm}%
\usepackage{algorithmicx}%
\usepackage{algpseudocode}%
\usepackage{listings}%
\usepackage{quiver}
\usepackage{tikz}
\usetikzlibrary{arrows}

\theoremstyle{thmstyleone}%
\newtheorem{theorem}{Theorem}
\newtheorem{proposition}[theorem]{Proposition}%

\newtheorem{lemma}[theorem]{Lemma}%

\theoremstyle{thmstyletwo}%
\newtheorem{example}{Example}%
\newtheorem{remark}{Remark}%

\theoremstyle{thmstylethree}%
\newtheorem{definition}{Definition}%

\newcommand{\concept}[1]{{\fontseries{b}\selectfont #1}}

\def\R{\mathbb{R}}
\def\Z{\mathbb{Z}}

\def\X{\mathbb{X}}

\def\C{\mathcal{C}}

\def\G{\mathcal{G}}
\def\T{\mathcal{T}}
\def\B{\mathbf{B}}

\def\ve#1{\mathchoice{\mbox{\boldmath$\displaystyle\bf#1$}}
{\mbox{\boldmath$\textstyle\bf#1$}}
{\mbox{\boldmath$\scriptstyle\bf#1$}}
{\mbox{\boldmath$\scriptscriptstyle\bf#1$}}}
\newcommand\va{{\ve a}}
\newcommand\vb{{\ve b}}
\newcommand\vc{{\ve c}}

\newcommand\vs{{\ve s}}
\newcommand\vg{{\ve g}}
\newcommand\vt{{\ve t}}
\newcommand\vq{{\ve q}}
\newcommand\vu{{\ve u}}
\newcommand\vv{{\ve v}}
\newcommand\vx{{\ve x}}
\newcommand\vy{{\ve y}}
\newcommand\vz{{\ve z}}
\newcommand\zero{{\ve 0}}
\newcommand\gmip{\G_{MIP}}

\DeclareMathOperator{\Adj}{Adj}
\DeclareMathOperator{\Compl}{Compl}

\begin{document}

\title[Article Title]{Augmentation approaches for Mixed Integer Programming}


\author[1]{\fnm{Justo Puerto}}\email{puerto@us.es}

\author[1]{\fnm{Jose Alberto Ruiz Alba}}\email{josealberto@gmail.com}

\affil[1]{ \orgname{IMUS},  \state{Sevilla}, \country{Spain}}


\abstract{This paper analyses the feasible sets structure of  general mixed integer linear programs (MIPs) and its relationship with the existence of a finite cardinality test set which can be applied in augmentation algorithms. We derive and characterize a computable, finite test set for MIPs which can be embedded in a finite augmentation algorithm. Several examples illustrate the structure of this set and its relationship with previous approaches in the literature.}



\maketitle

\section{Introduction}\label{sec1}

Improving directions algorithms are key in all areas of optimization, underlying many classical methods in both linear and nonlinear programming. Gradient-based methods in nonlinear optimization and the simplex method for linear programming are, in fact, two examples of improving directions approaches. In integer programming, augmentation (or primal) methods are also of current interest; see, for example, \cite{dLKS19, dLHK13}.
It is well known that universal test sets exist for continuous and pure integer linear programs, namely the set of circuits and the Graver basis, respectively. However, combining these two does not yield a finite universal test set for mixed-integer programs (MIPs), as shown by Hemmecke \cite{Hem03}. This raises the question of whether a finite universal test set exists for MIPs. In this paper, we answer this question positively by characterizing such a set.

Once a finite test set is found, an augmentation procedure can be applied: starting from a feasible solution, one iteratively improves it along an augmenting direction. This process enables one to find an optimal solution and to certify its optimality. While De Loera, Hemmecke and Lee \cite{dLHL15} established strong bounds on the number of augmentations required for linear and integer programs, their results do not extend to MIPs, where characterizing a finite set of augmenting directions is more challenging or remains unknown.

Several augmentation strategies for mixed-integer programming have been developed, notably by Hemmecke, Köppe and Weismantel \cite{KW03} and \cite{HKW}, and Hemmecke \cite{Hem03}. These works construct  “test sets” by exploiting geometric properties of polytopes. Since these seminal contributions, progress on this topic has been limited. However, recent works, such as Le Bodic et al. \cite{lBPPP18}, indicate a renewed interest. Other related studies include \cite{LPSX20, PWW20}, which elaborate on the distance between optimal solutions of mixed-integer programs and those of their linear relaxations. The first one focuses on pure integer programs, while the other addresses MIPs and improves the classic bound of \cite{CGST86}.

Our contribution is to develop an algebraic framework that yields an explicit finite universal test set for MIPs, together with an augmentation algorithm. We first study a restricted version of MIPs to identify the key structural issues, and then extend our construction to the general case via the notion of basic-integer solutions. Although the bounds in \cite{CGST86} do not apply in this setting, our results shed some light on the structure of Mixed Integer Programs.

The paper is organized as follows. Section 2 reviews known test sets for linear and integer programming and previous attempts for the mixed-integer case. Section 3 presents our construction in the restricted case and its extension to the general problem. Section 4 discusses the augmentation algorithm and complexity aspects, while Section 5 contains the main results: the finite universal test set for MIPs and the proofs of finiteness and universality. Section 6 concludes with remarks and perspectives for future research.

\section{Review of Results for Linear and Integer Programming} \label{se:review}
We follow the structure and terminology of de Loera, Hemmecke and Lee~\cite{dLHL15} in the presentation of these results.
Throughout this paper, we will work with a \concept{mixed integer program} (MIP) of the form 
\begin{equation} 
\min \{ \vc^T \vx ~|~ A \vx = \vb, ~ \zero \le \vx \le \vu, ~ \vx \in \X \}.\label{eq:program}
\end{equation}
Here $A \in \Z^{m \times n}$, $\vb \in \Z^m$, $\vc \in \Z^n$, $\X = \R^{n_R} \times \Z^{n_I}$ for some suitable $d$ and $n_R+n_I = n$.
In general, this represents a MIP.  If $n_R=0$, it represents an IP, and if $n_I=0$, an LP.  
We will often partition $A=(A^R , A^I)$, $\vc = (\vc^R, \vc^I)^T$, $\vx = (\vx^R, \vx^I)^T$ and $\vu = (\vu^R, \vu^I)^T$ 
into parts corresponding to the real and integer parts respectively.
The rows of $A^R$ can be taken to be independent after some preprocessing, so we consider that $A^R$ has rank $m$.

For programs of the form (\ref{eq:program}), we use the convention that an \concept{improving direction} is a vector $\vt$ such that $\vc^T \vt < 0$, and that we improve a vector by adding a scalar multiple of $\vt$ from the current solution.  That is, we move from vector $\vz$ to a vector $\vz+\alpha \vt$ for some scalar $\alpha > 0$.

A \concept{test set} for a mixed integer programming is a set of vectors such that for any non-optimal feasible solution, such an improving direction exists.
However, test sets for particular instances of (\ref{eq:program}) depend heavily on the details of the feasible set we are trying to explore, as well as the objective function. That leads us to the more robust notion of a \concept{Universal Test Set} for a given $A$, that is a set, $\T(A)$, which contains test sets for all possible choices of $\vb, \vc$ and $\vu$.
Note that the size of each of these sets may be exponential in the input size of $A$.

There are some particular cases that have been deeply studied:
\begin{enumerate}
    \item[$\bullet$] For LP, $T(A)$ is exactly the \concept{circuits} of $A$.  These are the support-minimal non-trivial solutions to $Ax=0$, normalized to have integer components and no common divisor. We denote the set of circuits of $A$ as $\C(A)$. 

    \item[$\bullet$] For IP, $T(A)$ is the \concept{Graver basis } of $A$.  This constructed as a union of Hilbert bases of $Ax=0$ over the intersection of $Ax=0$ with the $2^n$ co-ordinate orthants, and denoted $\G(A)$.  
    Recall that the \concept{Hilbert basis } of a pointed cone is the minimal set of vectors generating the integer points of the cone as non-negative integer combinations.
\end{enumerate}

A universal test set for mixed integer programs is described by Hemmecke~\cite{Hem03}, called therein the \concept{MIP Graver Test Set}.  The MIP Graver Test Set associated is determined by the presenting matrix $(A^R, A^I)$, and is denoted $\gmip(A^R, A^I)$ (or simply $\gmip(A)$). It consists of two parts. The first part is formed by extending the circuits of $A^R$ by appending $n_I$ zeros to give them the right dimension.  We call these \concept{padded circuits}. The second part consists of the vectors $[\vx^R, \vx^I]^T \in \ker_{\X}(A)$ with $\vx^I \ne \zero$ which are minimal in the partial order $\sqsubseteq$ (we say that $\vx \sqsubseteq \vy$ if they lie in the same orthant, i.e.~corresponding entries of $\vx$ and $\vy$ never have opposite signs, and $|\vx_i| \le |\vy_i|$ for each $i= 1 \ldots n$). Unfortunately, Hemmecke showed in \cite[pp. 257-258]{Hem03} that this test set is not finite in general.

\section{Augmenting Sets for MIP} \label{se:augmip}
A critical difference between MIP and the pure LP and IP contexts is that, contrary to the well-understood structure of their test sets (circuits and Graver basis, respectively), it is not well-characterized the existence of usable finite test sets for the MIP case, see for instance the examples of~\cite{CGST86, Hem03}.  However, it is known that one can find two finite sets that may support augmentation algorithms, although with intricate constructions.

One way to build such a pair of sets, suggested by Cook et al.~\cite{CGST86}, is to observe that we can restrict ourselves to moving among rational points with a certain combinatorial structure. 

For simplicity, we will ignore the upper bounds $\vx \le \vu$ in \eqref{eq:program}.
In some situations these are, in any case, superfluous.  If they are important, they can be handled either by working with extended basic solutions (see for instance \cite[page 220]{bazaraa04}) where non-basic $x_i$ may be either $0$ or $u_i$, or by reformulating the system through additional non-negative variables $\vy=\vu-\vx$. 

We then work with the \concept{standard form}
\begin{equation}
\min \{ \vc^T \vx ~|~ A \vx = \vb, ~ \zero \le \vx, ~ \vx \in \X=\R^{n_R} \times \Z^{n_I} \}. \tag{MIP} \label{eq:standard}
\end{equation}
Any solution to (\ref{eq:standard}) has the form $\binom{\vx^R}{\vx^I}$, where $\vx^I \in \Z^{n_I}$.  Then $A^R \vx^R + A^I \vx^I = \vb$, and the non-integer part $\vx^R$ is a solution to the following linear system with integer coefficients and right hand side: 
\begin{equation} 
    A^R \vx^R = \vb - A^I \vx^I, \qquad \zero \le \vx^R, ~~ \vx^R \in \R^{n_R}\label{eq:slice}
\end{equation}
An optimal solution $\vx_* = \binom{\vx^R_*}{ \vx^I_*}$ to (\ref{eq:standard}) consists of an integer point $\vx^I_*$ and a solution $\vx^R_*$ to the LP problem (\ref{eq:slice}). Hence, the point $\vx^R_*$ can be chosen to be a \concept{basic} solution to (\ref{eq:slice}) in the sense of linear programming, i.e.~we can restrict its support to a set of linearly independent columns. 
We assume $n_R>m$. This leads us to the following concepts.

\begin{definition}
A \concept{basic-integer solution} $\vx_*=(\vx_* ^R, \vx_* ^I)^T$ to (\ref{eq:standard}), is a solution whose support restricted to the real-valued variables corresponds to a linearly independent set of columns of $A^R$.  \label{de:bis}
\end{definition}
Note that, by Cramer's rule, the entries of ${\vx^R_*}$ must then be integer multiples of the reciprocal of a subdeterminant of $A^R$.

\subsection{Restricted Mixed Integer Programs}
In this section we focus our efforts on finding a finite universal test set for a certain type of MIP. Later on, we will use this sets to build another type of finite augmenting set for a general program. We start with our standard MIP
\begin{equation*}
\min \{ \vc^T \vx ~|~ A \vx = \vb, ~ \zero \le \vx, ~ \vx \in \X=\R^{n_R} \times \Z^{n_I} \}
\end{equation*}
and we consider a basis $B\subset A^R$ fixed. As an optimal solution of the problem is a basic-integer solution with basis $B$ (as it is the only one we are allowing), the non-basic part of the solution is $\zero$. Therefore, we end up with the following program
\begin{equation}
\min \{ \vc_B^T \vx_B + (\vc^I)^T \vx^I ~|~ B\vx_B+A^I\vx^I = \vb, ~ \zero \le \vx, ~ \vx \in \X=\R^{n_R} \times \Z^{n_I} \}. \tag{RMIP}\label{eq:restricted_MIP}
\end{equation}
We call this program a \concept{Restricted Mixed Integer Program}. Note that $\vx_B$ has only $m$ components. We consider the remaining $n_R-m$ components to be zero, so $\vx= (\vx_B, \zero, \vx^I)^T$.

As $B$ is an invertible matrix, we can rewrite $\vx_B$ on terms of $\vx^I$ as 
\begin{equation}
    \vx_B= B^{-1}(\vb-A^I \vx^I).\label{eq: xBDespejao}
\end{equation}
This will be the key to find a test set as we will focus on solving the integer part as if it were an integer program.

Given a basis $B \subset A^R$, consider $\vx_1=(\vx_1^R,\vx_1^I)^T$ and $\vx_2=(\vx_2^R, \vx_2^I)^T$ two distinct feasible solutions of the problem (\ref{eq:restricted_MIP}) and let us denote by $(B^{-1}A^I)_{p}$ the $p$-th row  of $(B^{-1}A^I)$. We define 
$$ S_B^+(\vx_1, \vx_2):= \{ p: (B^{-1}A^I)_{p}(\vx_2^I-\vx_1^I)\ge 0\}, \quad S_B^-(\vx_1, \vx_2):=\{1,\ldots,m\} \setminus S_B^+(\vx_1, \vx_2).$$

Now, we consider the rational cone 
$$ C_B(\vx_1, \vx_2):=\{\vz\in \mathbb{Z}^{n_I}: (B^{-1}A^I)_{p} \vz\ge 0,\; p\in S_B^+(\vx_1, \vx_2),\;  (B^{-1} A^I)_{p} \vz\le 0, p\in S_B^-(\vx_1, \vx_2)\}.$$
The convex hull of $C_B(\vx_1, \vx_2)$ might not be pointed, and in particular may not be contained in an orthant, see Example~\ref{ex:unpointed} for one example of this. In conclusion, we do not get the ``positive sum property'' directly. However, we can still get this by using the strategy that is used to build Graver bases, that is, decomposing the cone into orthants and finding a Hilbert basis for each orthant. 
Formally, let $\mathbb{O}^k$ be the $k$-th orthant of $\mathbb{R}^{n_I}$, $k=1,\ldots,2^{n_I}$. We consider the pointed integer cones $C_B^k(\vx_1, \vx_2)  = C_B(\vx_1, \vx_2) \cap \mathbb{O}^k$, for all $k=1,\ldots,2^{n_I}$.  
Each such cone has a Hilbert basis $H^k(\vx_1, \vx_2)$, and the union of these Hilbert bases gives a finite set such that any integer point in $C_B(\vx_1, \vx_2)$ can be represented by a non-negative integer linear combination of elements in some $H^k(\vx_1, \vx_2)$.

\begin{remark}
This construction is finite and it can be applied to an arbitrary rational integer cone $C$.  Given such a $C$, let $H^k(C)$ be the Hilbert basis for $C \cap \mathbb{O}^k$. We define the \concept{conic Graver basis} of $C$, $\G(C)$, as $\bigcup_{k=1}^{2^{n_I}} H^k(C)$.  This coincides with the standard definition of a Graver basis $\G(A)$ when $C$ is, in fact,  the integer points of a subspace given by the kernel of a matrix $A$, i.e., $C= \ker(A) \cap \Z^n$.
\end{remark}

We denote $\G_B( \vx_1, \vx_2)$ to the conic Graver basis $\G(C_B(\vx_1, \vx_2))$ above. Then, we can write
$$ \vx_2^I=\vx_1^I+\sum_{k \in K} \alpha_k \vg_k,\; \alpha_k \in \mathbb{Z}_+,\; \vg_k \in \G_B( \vx_1, \vx_2),\; \forall  k \in K.$$

\begin{remark}
In this context, given two vectors $\va$ and $\vb$, $\va \sqsubseteq \vb$ means that $a_i b_i \geq 0$ and $|a_i|< |b_i|$ for all $i$. Note that, by construction, for all $\vg \in \G_B(\vx_1, \vx_2)$ we get $\vg \sqsubseteq \vx^I_2- \vx^I_1$ and $B^{-1} A^I \vg \sqsubseteq B^{-1} A^I (\vx^I_2-\vx_1^I)$. \label{Conformalidad}
\end{remark}

Observe that $\G_B(\vx_1, \vx_2)$ depends only on the sign pattern of $B^{-1}A^I(\vx_2^I-\vx_1^I)$.  We can index these sign patterns by subsets of $[m]:=\{1, ..., m\}$ and the basis we are using, defining $C^S_B$ for $S \subseteq [m]$ and $B$ an invertible submatrix of $A^R$ to be
$$ C^S_B:=\{\vz\in \mathbb{R}^{n_I}: (B^{-1}A^I)_{p}\vz\ge 0,\; p\in S,\; ~  (B^{-1}A^I)_{p}\vz\le 0, p \in [m] \setminus S\}.$$

\begin{proposition} 
Given a basis $B\subset A^R$, any point $\vz^I \in \Z^{n_I}$ which is a difference of integer parts of basic-integer solutions can be written as $\vz^I = \sum_{k \in K} \alpha_{k} \vg_{k}, \; \alpha_{k}\in \mathbb{Z}_+, \vg_{k} \in \G(C_{B}^{S_B})$ for some $S_B \subseteq [m]$.
\end{proposition}

\begin{proof}
Any $\vz^I \in \mathbb{Z}^{n_I}$ induces a partition of the rows of $B^{-1}A^I$ into $S_B^+\cup S_B^-$.  We take $S_B=S_B^+$ in the above construction. 
\end{proof}

\begin{definition} Given a basis $B \subseteq A^R$, we introduce the following sets.
    \begin{description}
\item[· \,$\G^{A,\vb}_{B}$ ]
 Let $\G^{A,\vb}_{B}$ be the union of the conic Graver bases $\G_B(\vx_1, \vx_2)$ formed from particular pairs of basic-integer solutions to (\ref{eq:restricted_MIP}) where $\vx_1$ and $\vx_2$ are basic-integer  solutions associated with the basis $B$. \label{de:gab}
 \item[· \,$\G_B^{A}$]
Let $\G_B^{A}$ be the union of the conic Graver bases $\G_B(\vx_1, \vx_2)$ over all cones $C_B(\vx_1, \vx_2)$ formed from particular pairs of basic-integer solutions to  (\ref{eq:restricted_MIP}) ranging over all integer vectors $\vb$.  \label{de:ga}
\item[· \,$\G^*_B$] 
Let $\G^*_B$ be the union of the conic Graver bases $\G(C_{B}^S)$ over all cones $C_{B}^S$ for $S \subseteq [m]$.\label{de:gstar}
 \end{description}

\end{definition}

\begin{remark}
    For better coherence with the following procedure for standard MIPs, we will consider $\vx_1$ to be a solution of the unrestricted problem. In this section this will be irrelevant as it only adds extra elements.
\end{remark}

We get that, for any $(A,\vb)$, we have $\G_B^{A,\vb} \subseteq \G_B^A \subseteq \G_B^*$. Now, we note that $\G_B^*$ is finite by construction.

We do not necessarily have that all $C_{B}^S$ correspond to an actual pair of feasible solutions $(\vx_1, \vx_2)$, but given such a pair we can always obtain a finite positive combination such that $\sum_{k \in K} \alpha_k \vg_k= \vx_2^I-\vx_1^I$ where $\vg_k \sqsubseteq \vx_2^I-\vx_1^I$ and  $\vg_k\in \G(C_B^S)$ for some $S$, for all $k \in K$. Now, using \eqref{eq: xBDespejao} we obtain:

\begin{equation*}
    \begin{split}
        (\vx_2)_B-(\vx_1)_B 
        &= B^{-1}(\vb-A^I \vx_2^I)-B^{-1}(\vb -A^I \vx_1) \\
        &=-B^{-1}A^I(\vx^I_2-\vx_1^I) \\
        &= -B^{-1}A^I(\sum_{k \in K}\alpha_k \vg_k) \\
        &= \sum_{k \in K}- \alpha_k B^{-1}A^I\vg_k.
    \end{split}
\end{equation*}
Thus, it follows that:
\begin{equation}
    \vx_2 - \vx_1 =
    \left[ \begin{array}{c}
        (\vx_2)_B- (\vx_1)_B  \\
        \zero \\ 
        \vx_2^I-\vx_1^I
    \end{array} \right]
    = \sum_{k \in K} \alpha_k 
    \left[ \begin{array}{c}
        -B^{-1}A^I \vg_k  \\
        \zero \\ 
        \vg_k
    \end{array} \right]\label{eq: aparecenT}
\end{equation}
for some $\vg_k \in \G_B(\vx_1, \vx_2)$.

\begin{definition}
    Given a point $\vg \in \Z^{n_I}$ and a basis $B \subset A^R$, we call the vector 
       $$ \left[ \begin{array}{c}
        -B^{-1}A^I \vg  \\
        \zero \\ 
        \vg
    \end{array} \right]$$
    a \concept{basic-integer direction} of basis $B$ and integer part $\vg$. Given a set of points $G \subseteq \Z^{n_I}$ we call $\T_B(G)$ to the set of all basic-integer directions of basis $B$ and integer part $\vg \in G$. \label{def:bid}
\end{definition}

\begin{remark}
    In our scenario, we consider $G$ as one of the sets of elements of conic Graver bases previously defined, so, in order to keep the notation simple, we will reduce the number of indices whenever possible. For instance, $\T_{B}(\G(\vx_1,\vx_2))$ will be shorten to $\T_B(\vx_1, \vx_2)$ where $B$ is a basis of $A^R$.
\end{remark}

It is straightforward to see that these basic-integer directions lie on $\ker_{\X}(A)$. Moreover, if $\vc^T \vx_2 < \vc^T \vx_1$, there is always an \concept{improving} basic-integer direction in $\T_B(\vx_1, \vx_2)$. 

The discussion above leads to the following result.
\begin{lemma}
    Given a basis $B\subset A$, consider two distinct basic-integer feasible solutions for \eqref{eq:restricted_MIP}, $\vx_1$ and $\vx_2$. There exist integer directions in $\T_B(\vx_1, \vx_2)$, i.e., vectors of the form 
        \[\vt=
        (\left[ \begin{array}{c}
        -B^{-1} A^I   \\
        \zero \\ 
        Id_{n_I}
        \end{array} \right]
        \vg)\] 
   for $\vg \in \G(\vx_1, \vx_2)$, verifying that $\vx_1+\vt$ is feasible. In addition, if $\vc^T\vx_2<\vc^T\vx_1$ then there always exists a vector $\vt_{0}$ of the previous form such that $\vc^T(\vx_1+\vt_{0})< \vc^T \vx_1$, in other words, there always exists an improving integer direction, hence, $\T_B^{A,b} := \T_B(\G_B^{A,b})$ is a test set for \eqref{eq:restricted_MIP}.\label{lemma: restricted_test_set}
\end{lemma}

\begin{proof}
    As in \eqref{eq: aparecenT}, we can write $\vx_2-\vx_1$ as $\sum_{k \in K} \alpha_k \vt_k$. From this, there must exist a $k_0$ such that $c^T(\vt_{k_0})<0$. 

    Let us check the feasibility of $\vx_1+\vt_k$ for all $k \in K$ with $\alpha_k \neq 0$.
    We know that 
    $$\vt_k= [-B^{-1}A^I, \zero, Id_{n_I}]^T \cdot \vg_k$$
    where $\vg_k$ satisfies $\vg_k \sqsubseteq \vx_2^I-\vx_1^I$ and $B^{-1}A^I \vg_k \sqsubseteq B^{-1}A^I (\vx^I_2-\vx^I_1)$. With this in mind, we can prove that 
    $$\vx_1+ \vt_k= 
    \left[ \begin{array}{c}
    B^{-1}(\vb- A^I \vx_1^I)   \\
    \zero \\ \vx_2^I
    \end{array} \right] +
    \left[ \begin{array}{c}
    -B^{-1}A^I   \\
    \zero \\ Id_{n_I}
    \end{array} \right] \vg_k 
    \geq \zero.$$
    In fact, the proof of the integer part and the continuous part are identical. We will show that $B^{-1}(\vb-A^I\vx_1^I)-B^{-1}A^I\vg_k \geq \zero.$

    First, if $B^{-1}A^I \vg_k \leq \zero$ the result follows as $(\vx_1)_B= B^{-1}(\vb -A^I \vx_1^I) \geq \zero$. Let us assume that $B^{-1}A^I \vg_k \nleq \zero$, which means that there is a $j\in [m]$ such that $(B^{-1}A^I \vg_k)_j>0$. Thanks to \ref{Conformalidad}, this tells us that $(B^{-1}A^I(\vx_1^I-\vx^I_2))_j$ and $(B^{-1}A^I \vg_{k'})_j$ for all ${k'} \in K$ with $\alpha_i \neq 0$ are non-negative numbers. As we know, 
    $$(B^{-1}A^I \vx_2^I)_j = (B^{-1}A^I \vx_1^I)_j + \sum_{k' \in K} (B^{-1}A^I \vg_{k'})_j,$$
    hence $(B^{-1}A^I\vx_2)_j \geq (B^{-1}A^I (\vx_1^I+\vg_k))_j$,
    and by 
    hypothesis we have that ${(B^{-1} \vb)_j \geq (B^{-1}A^I \vx_2^I)_j}$, so 
    $$(B^{-1} \vb)_j \geq (B^{-1}A^I \vx_2^I)_j \geq (B^{-1}A^I(\vx_2+\vg_k))_j $$
    and the result follows. 
\end{proof}

\begin{remark}
    The above result can be seen as a consequence of corollary of \cite[\S\S  3.3]{Hem03} when the real matrix is square. Nevertheless, the simplicity of this restricted problem allows us to prove it as if it were an integer program using Graver bases.
\end{remark}

As a consequence of Lemma \ref{lemma: restricted_test_set} and the fact that $\mathcal{G}_B^{A,\vb}\subset \mathcal{G}_B^*$ for all $\vb$, $\mathcal{T}_B^*$ is actually an universal test set for the Restricted Mixed Integer Programs considered in this section. 

Once we have understood the structure of the test sets for RMIPs, we address a similar question for a general MIP in the next section.

\subsection{Mixed Integer Programs}
Let us consider a general MIP given by
\begin{equation}\tag{\ref{eq:standard}}
\min \{ \vc^T \vx ~|~ A \vx = \vb, ~ \zero \le \vx, ~ \vx \in \X=\R^{n_R} \times \Z^{n_I}\}
\end{equation}
where $B_1, ..., B_r$ are all the bases of $A^R$. 
Given two basic-integer solutions, we define the cone $C(\vx_1, \vx_2)$ as the intersection of all the cones $C_{B_i}(\vx_1, \vx_2)$ for all $i \in \{1, ..., r\}$:
$$C(\vx_1, \vx_2) := \big\{ \vz \in \Z^{n_I} \, | \,  \vz \in C_{B_i}(\vx_1, \vx_2) \text{ for all } i \in \{1,..., r\} \big\}.$$
This cone is non empty, as $\vz= \vx_2-\vx_1$ belongs to every single cone $C_{B_i}(\vx_1, \vx_2)$.
As we did in previous sections, we can compute this cone by just looking at the sign of $B_i^{-1}A^I(\vx_2-\vx_1)$ for all $i$. More generally, given a list $S=\{S_1, ..., S_r\}$ with $S_i \subseteq [m]$ for every $i \in \{1,..., r\}$, we define
$$C^S := \big\{ \vz \in \Z^{n_I} \, | \, \vz \in C^{S_i}_{B_i} \text{ for all } i \in \{1,..., r\} \big\}.$$

\begin{definition}
    We define the following sets:
    \begin{description}
        \item[· \,$\G^{A,b}$]
Let $\G^{A,b}$ be the union of all conic Graver bases of the cones $C(\vx_1, \vx_2)$ for every pair of basic-integer feasible solutions of \eqref{eq:standard}. In addition, we denote $\T^{A,b}$ to the union of all the sets $\T_{B}(\G^{A,b})$ for all basis $B \subset A^R$. \label{def: gab}

\item[· \,$\G^{A}$]  Let $\G^{A}$ be the union of all sets of the form $\G^{A,b}$ varying $\vb$ over $\Z^{n_I}$. In addition, we denote $\T^{A}$ to the union of all the sets $\T_{B}(\G^{A})$ for all basis $B \subset A^R$.\label{def: ga}

\item[· \,$\G^*$] 
    Let $\G^*$ be the union of all conic Graver bases of the cones $C^S$ for all list $S=[S_1, ..., S_r]$ with $S_i \subseteq [m]$. In addition, we denote $\T^*$ to the union of all the sets $\T_B(\G^*)$ for all basis $B\subset A^R$. \label{def: g*}
\end{description}
\end{definition}

Unfortunately, $\T^*$ is clearly not a test set for \eqref{eq:standard} as we do not have a clear way of moving between basic-integer solutions from different bases with the same integer part. One can try to avoid it by adding the (padded) circuits to our set. However, this strategy is not enough as it will be shown in the following examples:

\begin{example} \label{ex: lone_basis}
We consider the following program with two real and one integer variables: 
$$\{\min \vc^T \vx ~|~ x_1 - x_2 + 2 x_3 = 3, x_1,~ x_2 \in \R_+, ~ x_3 \in \Z_+\}.$$
We have $A^R = [1 ~ -1]$ and $A^I = [2]$.  Then, there are two bases in $A^R$, call them $B_1=1$ and $B_2=-1$ corresponding to the choice of the first and second columns respectively.  There is a pair of solutions where $B_1$ is basic, namely $\vx_0=(3,0,0)^T$ and $\vx_1=(1,0,1)^T$ and a family of them corresponding to $B_2$ that we name $\vx_b=(0,2b-3,b)^T$ for $b\in \{ 2,3,\ldots \}$.

We can easily see that $\G^{A,b}=\G^A=\G^*= \{\pm1 \}$ by only looking at the cones $C(\vx_0, \vx_1)$ and $C(\vx_1, \vx_0)$:
$$C(\vx_0, \vx_1)= \{z \in \Z ~|~ 2z \geq 0, -2z \leq 0\}= \Z_{\geq 0},$$
$$C(\vx_1, \vx_0)= -C(\vx_0, \vx_1)= \Z_{\leq 0}.$$
Therefore, $\T_{B_1}(\G^*)=\{\pm (-2,0,1)^T\}$ and $\T_{B_2}(\G^*)=\{ \pm(0,2,1)^T\}$.

The problem appears when we try to move from a basic-integer solution of basis $B_1$ to a basic-integer solution of basis $B_2$. For instance, in order to move from $\vx_2$ to $\vx_1$ we need the direction $(1,-1,-1)^T$ which is not in $\T^*$ neither it is a circuit. In this case, it is a combination of both.

\begin{table}[ht]
    \centering
    \caption{Feasible set of the example \ref{ex: lone_basis}. Unnamed points are infeasible .}
    \label{T: E5.3}
    
    \begin{tabular}{|c|c|}
        \hline
        \textbf{$B=1$} & \textbf{$B=-1$} \\
        \hline
        $\vx_0 = \left[ \begin{array}{c}
        3   \\0 \\ 0
        \end{array} \right]  \rule{0pt}{12pt}$ &  
        $ \left[ \begin{array}{c}
        0  \\-3 \\ 0
        \end{array} \right] $ \\  
        \hline
        $\vx_1 = \left[ \begin{array}{c}
        1   \\0 \\ 1
        \end{array} \right]  \rule{0pt}{12pt}$ &  
        $\left[ \begin{array}{c}
        0  \\-1 \\ 1
        \end{array} \right]  $ \\  
        \hline
        $\left[ \begin{array}{c}
        -1   \\0 \\ 2
        \end{array} \right] \rule{0pt}{12pt} $ &  
        $\vx_2 = \left[ \begin{array}{c}
        0  \\1 \\ 2
        \end{array} \right]$ \\  
        \hline
        $ \begin{array}{c}
       .  \\. \\ .
        \end{array}$ &  
        $\begin{array}{c}
        .  \\. \\ .
        \end{array}$ \\  
        \hline
    \end{tabular}
\end{table}
\end{example}

\begin{example} 
A variation of the previous Example~\ref{ex: lone_basis}, with two real and two integer variables. Consider now a problem of the form $$\{\min \vc^T \vx ~|~ x_1 - x_2 + 2 x_3 + 2x_4 = 1, x_1, x_2 \in \R_+, x_3, x_4 \in \Z_+\}.$$
We have $A^R = [1 -1]$ and $A^I = [2, 2]$. There are only two feasible bases, $B_0=1$ and $B_1=-1$.  Solutions include $\vx_0=(1,0,0,0)^T$ and $\vx_{u,v}=(0,2u+2v-1,u,v)^T$ for $u,v \geq 0$.

If we do some calculations we get:
\begin{equation*}
    \begin{split}
        C(\vx_{u,v}, \vx_0)= & \{\vz \in \Z^2 \, | \quad (2,2)^T \vz \leq \zero, \, -(2,2)^T\vz \geq \zero \} \\
        C(\vx_0, \vx_{u,v}) = & -C(\vx_{u,v}, \vx_0)\\
   \end{split}
\end{equation*}
Hence, $\G(C(\vx_{u,v}, \vx_0))=\{ (-1,0)^T, (0,-1)^T, \pm(1,-1)^T\}=-\G(C(\vx_0, \vx_{u,v}))$. Then, 
$$\big \{\pm(1,0)^T, \pm(0,1)^T, \pm(1,-1)^T \big \} \subseteq \G^{A,b}.$$
Furthermore, it can be checked that 
$$ \big \{\pm(1,0)^T, \pm(0,1)^T, \pm(1,-1)^T \big \} = \G^{A,b}= \G^A= \G^*.$$ Thus,
\begin{equation*}
    \begin{split}
        \big \{ \pm(-2, 0, 1, 0)^T, \pm(-2, 0, 0, 1)^T,
       \pm (0, 2, 1, 0)^T, \pm(0, 2, 0, 1)^T, \pm(0, 0, 1, -1)^T\big \} = \T^{A,b}
    \end{split}
\end{equation*}
This example is a generalization of the previous one. It presents the same problematic, $\C(A^R) \cup \T^*$ is not a test set. \label{ex:unpointed}
\end{example}

The above examples suggest that we should look beyond the current framework (circuits and basic-integer directions). With this in mind, we introduce a new concept.

\begin{definition}
    Given a standard Mixed Integer Program \eqref{eq:standard}, a pair $(\T, \C)$ is called a \concept{Double Test Set} if for any feasible non-optimal solution $\vx$ of \eqref{eq:standard} there exists a direction $\vt_0 \in \T$ such that there exist $\vs_1, ..., \vs_\ell \in \C$ and non negative real numbers $\beta_1,\ldots,\beta_\ell$, verifying that $\bar \vx =\vx + \vt_0+ \sum_{i=1}^\ell \beta_i \vs_i$ is a feasible solution which improves $\vx$, i.e.,  $\vc^T \bar \vx< \vc^T \vx$. We say that a Double Test Set is \concept{universal} if it is a Double Test set for any choice of $\vb$ and $\vc$.\label{def: DTestSet}
\end{definition}

Observe that combining elements from the two sets to form a universal test set may result in an infinite set since the number of linear combinations of elements of $\mathcal{C}$ may not be finite.
\begin{theorem}
    The pair $(\T^*, \C(A^R))$ is a finite universal double test set for a standard MIP problem.\label{Th: DTS}
\end{theorem}
The reader may note that finite cardinality of the double test set does not imply a finite number of improving directions due to the infinite number of linear combinations of circuits that may be required.

In order to prove this theorem we use the following construction. 
Consider the set:
\begin{equation*}
    \begin{split}
        \gmip(A)= & \big \{(\vq,\zero)^T, \mbox{ for } \vq \in \C(A^R) \big \} \, \, \cup \\
        & \big \{(\vq, \vz)^T \in \ker_{\X}(A):  \, \mbox{ there is no } (\vq', \vz')^T \in \ker_\X(A) \mbox{ such that } (\vq', \vz') \sqsubseteq (\vq, \vz) \big \}
    \end{split}
\end{equation*}
    
By \cite[Lemma 12]{Hem03}, this is a universal test set for \eqref{eq:standard}. Nevertheless, this set can be infinite, as shown in example \cite[pp. 257-258]{Hem03}. To avoid this drawback, first we define the projection onto the integer part:
\begin{equation*}
    \begin{split}
        \phi (\R^{n_R}\times \Z^{n_I})  &\longrightarrow \Z^{n_I} \\
        (\vq, \vz) & \mapsto \vz.
    \end{split}
\end{equation*}
In addition, for any polyhedron 
$$P_{\vu}:=\{ \vx \in \R^{n_R} \, | \, A^R \vx = -A^I \vu \}$$
we consider the relationship $(\vg, P_{\vg}) \sqsubseteq (\vz, P_{\vz})$ that holds if
 $\vg \sqsubseteq \vz$ and $P_{\vz}= P_{\vg} + P_{\vz- \vg}$. (The interested reader can find  some results on how to determine if $P_{\vg} \sqsubseteq P_{\vz}$ in \cite{HKW}.)

Then, we consider a finite set $G$ such that for every element $\vz \in \phi(\ker_{\X}(A))$ either $\vz \in G$ or if $\vz \not \in G$, $\vz$ can be written as $\vz = \sum_{i=1}^k \alpha_i \vg_i$  with $\alpha_i \in \Z_{>0}$ for some elements $\vg_1, ..., \vg_k \in G$ such that $(\vg_i, P_{\vg_i}) \sqsubseteq (\vz, P_{\vz})$ for all $i\in \{1,..., k\}$. One can see this set as an analogue to a Gröbner basis where a monomial $\vx^\vv$ reduces $\vx^\vu$ if $(\vv, P_\vv) \sqsubseteq (\vu, P_\vu)$. By \cite[Proposition 1]{Hem03} the set $\phi(\gmip(A))$ is contained in $G$.

With this machinery, we are in position to state a technical lemma that will be used in the proof of Theorem \ref{Th: DTS}.

\begin{lemma}[Corollary 4.\cite{Hem03}]
    Given a matrix $A=(A^R, A^I)$, let $B_1, ..., B_r$ be all the bases of $A^R$. For any $\vz \in \Z^{n_I}$, consider the function 
    $$f(A,\vz) = ( \vz, B_1^{-1}A^I\vz, ..., B_r^{-1}A^I\vz).$$
     Then, $f(A, \vz') \sqsubseteq f(A, \vz)$ implies $(\vz', P_{\vz'}) \sqsubseteq (\vz, P_\vz)$.
\end{lemma}

We note that, based on the lemma above, an element $\vz \in C^S$ is in the conic Graver basis of $C^S$ if and only if there is no other element $\vz' \in \Z^{n_I}$ such that $f(A, \vz')\sqsubseteq f(A, \vz)$.

With this result, we can ensure the adequate structure of our Double Test Set: for every augmentation $(\vq,\vz)^T$ of $\gmip(A)$, we know that $\vz \in G$ so we can find a direction $\vt \in \T^*$ such that the difference between the element $(\vq, \vz)^T$ and $\vt$ belongs to $\langle\C(A^R) \rangle$.

\begin{proof}[Proof of Theorem \ref{Th: DTS}]
    Given the set $\gmip(A)$ and a feasible non-optimal solution $\vx$ of \eqref{eq:standard}, there always exists a vector $\vt=(\vq, \vz)^T \in \gmip(A)$ such that $\bar \vx= \vx + \vt$ is feasible and an improvement over $\vx$. Our goal is to prove that there exists a combination of vectors of our finite set $(\T^*, \C(A^R))$ that allows us to move from $\vx$ to $\bar \vx$. It is sufficient to prove that $\vz \in \G^{A,b}$ as we can solve any LP problem on matrix $A^R$ using the padded circuits.
    
    If $\vz= \zero$ then $\vt \in \C(A^R)$ and we consider $\vt_0= \zero$.

    Let us suppose that $\vz \neq \zero$. By the previous lemma, we know that $f(A,\vz)$ is $\sqsubseteq$-irreducible by elements of $\phi(\ker_\X(A))$, as if not, there would exist a $\hat \vz$ which would satisfy $(\hat \vz, P_{\hat \vz})\sqsubseteq (\vz, P_{\vz})$ and $\vz$ would not be in $G$. Now, by contradiction, assume $\vz \not \in \G^{A,b}$, but $\vz \in C(\bar \vx, \vx)$ as $\vz = \bar \vx^I - \vx^I $, thus, we get that there is an element $\vz' \in \G^{A,b}$ such that $f(A,\vz') \sqsubseteq f(A,\vz)$. This is a contradiction since $\vz' \in \phi(\ker_\X(A))$.
\end{proof}

\section{The augmentation algorithm\label{se:aug-algo}}
In this section we present 
the natural algorithm to solve mixed integer programming problems of the form
\begin{equation*}
    \tag{MIP}
    \min \{ \vc^T \vx ~|~ A \vx = \vb, ~ \zero \le \vx, ~ \vx \in \X \}
\end{equation*}
provided that a double test set of the form $(\T, \C)$ is available. 

As expected for an augmentation algorithm, we need to start with an initial solution. Let us consider $\vx_0$, a feasible solution of the problem. 

In fact, at any step of the algorithm, we may assume that our solution is the optimal solution of the problem on the feasible set given by \eqref{eq:slice}, as we can solve this program using padded circuits.

\begin{algorithm}
    \caption{Augmentation Algorithm}\label{alg: AA}
    \begin{algorithmic}[1]
        \State \textbf{input} a MIP problem like \eqref{eq:program}, a double test set $(\T, \C)$ and  $\vx_0$ a feasible solution of \eqref{eq:program}.
        \State \textbf{output} an optimal solution for \eqref{eq:program}.
        \While{there are $\vt \in \T$ and $\vs_1, ..., \vs_\ell$ such that $\vx_0 +\vt + \sum_{i=1}^\ell \beta_i \ell_i$ is feasible and improvement over $\vx_0$}
            \State $\vx_0 \gets \vx_0+\vt+\sum_{i=1}^\ell \beta_i \vs_i$
        \EndWhile
    \end{algorithmic}
\end{algorithm}

 \begin{remark}
    If $\T^{A,b} \subseteq \T$, for a given solution $\vx$ of basis $B$, we do not need to search in the whole set, as with the set $\T$ we are focusing on modifying the integer part so we can later recover feasibility using $\C(A^R)$ (or a LP-solver). In addition, all the sets $\T^{A,b}_B$ share the same integer parts, therefore, if $\vx_0$ satisfies the optimality condition we can restrict ourselves to just work with elements $\vt$ in $\T_B^{A,b}$ such that $\vc^T \vt < \zero$. The reason is that all the directions of $\T_B^{A,b}$ do not change the optimality condition, hence $\vx_0 + \vt$ also verifies it. If this new point is feasible we have found an improvement, if not we need to recover the feasibility by worsening the solution using some padded circuits.
\end{remark}

\subsection{About the complexity and the number of steps}

It is known that, in general, MIPs and IPs are NP-hard, so we cannot expect to solve these problems in polynomial time, nor even using augmenting algorithms. In fact, we are far from polynomiality, as calculating each Hilbert basis gives us a lower bound on the complexity that is exponential in the input size of $A$.

Furthermore, the main bottleneck in this approach is the combinatorial structure of the problem. We do not know if a basis is the optimal one without exploring it. This tells us that the number of steps of this Augmentation Algorithm is, in the worst case, exponential as it will need to visit all feasible bases.

For that reason, we focus on the algebraic part of the algorithm. Nowadays, with the current state of the art, it is not feasible to calculate the set $\G^{A,b}$ nor $\G^A$ without exploring all pair of feasible solutions, hence, we are led to work with the potentially much bigger set $\G^*$, which , on the other hand, is easier to compute.

\section{Building a finite test set\label{se:finite-universal}}
In the previous two sections we have presented a finite double \concept{universal} test set. The reason why we need two sets instead of one is that in order to find the correct combination of circuits and integer directions we need to fix $\vb$. In this case, we can find a finite test set that may not be universal.

There is not a unique way to find such a set. For instance, one can consider all the possible differences between all basic-integer feasible solutions. The downside of this method is obvious: we need to calculate every single basic-integer feasible solution. However, we can do it in an alternative way.

Consider a basic-integer feasible solution $\vx_0$ of \eqref{eq:standard} of basis $B_0$ and a finite set $G$ containing $\phi(\gmip(A))$. If $\vx_0$ is not an optimal solution there must exist a direction ${\vt=(\vq, \vz)^T \in \gmip(A)}$ such that $\bar \vx= \vx_0 + \vt$ is feasible and an improvement over $\vx_0$. If $\bar \vx$ is not a basic-integer solution, we know that we can find one that improves $\bar \vx$ using padded circuits, therefore, we assume $\bar \vx$ to be a basic-integer solution. Let $\vt_0$ be
the basic-integer direction of integer part $\vz$ and basis $B_0$. Our objective is to find the difference $\vt- \vt_0$. By construction, $\vt- \vt_0$ lies in the ideal $\langle \C(A^R) \rangle$ (padded). We can find a set that contains such a difference in an algorithmic way. For that, we need the following definition:

\begin{definition}
    Given a basic-integer solution $\vx$ of basis $B\subset A^R$, we call the \concept{set of adjacent} (or neighbors) basic-integer solutions to $\vx$, and we denote it $\Adj(\vx)$, to the set of basic-integer solutions with the same integer part as $\vx$ and any basis $B' \subset A^R$ that shares $m-1$ columns with $B$.
\end{definition}

\begin{algorithm}
    \caption{Completion Procedure: using circuits}\label{alg: Buscando q}
    \begin{algorithmic}[1]
        \State \textbf{input} $\vx_0=(\vx^R_0, \vx^I_0)^T$ a basic-integer feasible non optimal solution of \eqref{eq:standard}, $\vg$ the integer part of an improvement $(\vq, \vg)^T$ of $\vx_0$, and the set of circuits $\C(A^R)$.
        \State \textbf{output} $T$, a set containing $(\vq, \vg)^T$.

        \State $\vt_0 \gets $ a basic-integer direction of integer part $\vg$ and a basis of $\vx_0$.
        \State $T \gets \{ \vt_0 \}$
        \State $\bar \vx \gets \vx_0 + \vt_0$
        \State $C \gets \Adj(\bar \vx)$

        \While{$C \neq \emptyset$}
            \State $\bar \vx \gets$ an element in $C$
            \State $C \gets C \setminus \{\bar \vx \}$
            \State $\vt= \bar \vx - \vx_0$
            \If {$\vt \not \in T$} 
                \State $T \gets T \cup \{\vt \}$
                \State $C \gets C \cup \Adj(\bar \vx)$
            \EndIf
        \EndWhile
    \end{algorithmic}
\end{algorithm}
\begin{lemma}
    Assuming the situation of the previous description, Algorithm \ref{alg: Buscando q} terminates and satisfies its specifications.
\end{lemma}
\begin{proof}
    First, we are going to show that the algorithm terminates. For that, note that all possible directions $\vt$ generated by the algorithm have the same integer part. With this in mind and using the definition of the set of adjacents to a basic-integer solution, we only get directions that move $\vx_0$ to a new basic-integer solution. As those sets are finite and we only branch each solution once, the procedure terminates and outputs a finite set. 

    It remains to prove that $(\vq, \vz)^T \in T$. Again, $\vx_0 + \vt_0$ gives a vector basic-integer solution that has the same integer part as $\bar \vx=\vx_0 + (\vq, \vz)^T$. By construction, the padded circuits transform $\vx + \vt$ into all the basic-integer solutions with the same integer part, in particular, $\bar \vx=\vx_0 + (\vq, \vg)^T$.
\end{proof}

To summarize, we have a strategy that given certain set $G$ and a feasible solution $\vx_0$ returns a set which contains an improving direction for a specific $\eqref{eq:standard}$ (if possible). If we repeat this approach with all possible basic-integer solutions the result is a finite test set for $\eqref{eq:standard}$.

\subsection{Degeneracy}
In the third section we show that the set $\T^* \cup \C(A^R)$ is not, in general, a test set for a MIP. Here, we elaborate a bit more on this issue. We have the following scenario: there is a solution $\vx_0$ of basis $B_0$ which is not optimal, hence, there is a direction $\vt=(\vq, \vz)^T$ such that $\bar \vx = \vx_0 + \vt$ is an improvement over $\vx_0$. We assume that $\vx_0$ cannot be improved using padded circuits and that $\vx_0 + \vt_0$ is not feasible, where $\vt_0$ is the basic-integer direction of integer part $\vz$ and basis $B_0$. 

\begin{definition}
    Given a point $\vx_0= (\vx_0^R, \vx_0^I)^T$ that verifies $A \vx = \vb$, we say that it is \concept{degenerated} if $\vx^R$ is a basic solution of the program $A^R \vx = \vb -A^I \vx_0^I$ for more than one basis $B \subset A^R$.
\end{definition}

In the following, we consider all basic-integer solutions to be non-degenerated. If a basic-integer solution is degenerated then there are more than one basic-integer direction associated with it and we consider the solution to be written in the basis that gives the direction that is actually relevant. With this in mind, we assume that all degenerated points $(\vx_0^R, \vx_0^I)$ verify $\vx_0^I \not \in \Z^{n_I}$. 
Moreover, in what follows, we call dual simplex directions any direction obtained by an iteration of the dual simplex algorithm that moves between non-feasible points that satisfy the optimality condition of the problem.

In the proof of the following result, we use these dual simplex directions.  Thus, given an integer part $\vx^I$, our primal program will be 
$$ \min \{\vc_R^T \vx \quad |  \quad A^R \vx = \vb - A^I \vx^I, \, \vx \geq \zero\}.$$

\begin{lemma}
    Under the hypothesis of the above discussion, if $\vt = (\vq, \vz) =\vt_0 + \sum_{k=1}^\ell \lambda_k \vs_k$ where $\vs_k$ are Dual Simplex directions, then $\vt$ can be written as a convex combination of basic-integer directions of integer part $\vz$. \label{L: CombConv}
\end{lemma}

\begin{proof}
    Let $\vx'_i$ be $\vx_0 + \vt_0 + \sum_{k=1}^i \lambda_k \vs_k$ for $i \in \{0, ..., \ell\}$ where $\vx'_\ell= \bar \vx$. By hypothesis, all the points $\vx_i'$ for $ i \in \{0,..., \ell-1 \}$ are non feasible. Let $B_i$ be the basis of each $\vx'_i$. This leads to the following diagram
   \[\begin{tikzcd}
	{\vx_0'} & {\vx_1'} & {...} & {\vx_{\ell-1}'} & {\bar \vx} \\
	\\
	{\vx_0} & {\vx_1} & {...} & {\vx_{r-1}} & {\vx_\ell}
	\arrow["{{\vs_1}}", from=1-1, to=1-2]
	\arrow["\begin{array}{c} \begin{array}{c} \\ \vs_2 \end{array} \end{array}", from=1-2, to=1-3]
	\arrow[from=1-3, to=1-4]
	\arrow["{{\vs_\ell}}", from=1-4, to=1-5]
	\arrow["{\vt_0}"', from=3-1, to=1-1]
	\arrow[from=3-1, to=3-2]
	\arrow["{{\vt_1 }}"', from=3-2, to=1-2]
	\arrow[from=3-2, to=3-3]
	\arrow[from=3-3, to=3-4]
	\arrow["{{\vt_{\ell-1}}}"', from=3-4, to=1-4]
	\arrow[from=3-4, to=3-5]
	\arrow["{{\vt_\ell}}"', from=3-5, to=1-5]
\end{tikzcd}\]
    where $\vt_i$ is the basic-integer direction of integer part $\vz$ and basis $B_i$, and $\vx_i$ is the basic-integer solution $\vx'_i - \vt_i$.

    We know that every movement $\vs_i$ takes out a negative component of the solution and it introduces a non negative new one. Let us say, without loss of generality, that $\vs_i$ takes out the $i_-$ component and introduces the $i_+$-th one. As we assumed that all basic-integer solutions are non-degenerated, we can also assume that $\vs_i$ makes the $i_+$ coordinate positive. 
    
    With all this in mind, $(\vx_0')_{0_-}$ is negative, but $\vx_0$ is feasible so the direction $\vt_0$ must have a negative $0_-$-component. Therefore, by continuity, there exists an $\alpha_0 \in [0,1]$ such that $$\vx_{(0,1)}=\vx_0 + \alpha_0 \vt_0$$
    has a zero in the $0_-$-component. As $\vx_{(0,1)}$ has a new zero, it is a degenerated point (if it is a basic-integer solution, either $\vx_{(0,1)}$ is $\vx_0$ or $\bar \vx$ and we can move on). Here, $\vx_{(0,1)} + (1-\alpha_0) \vt_1$ is a basic-integer solution of basis $B_1$ (as the $0_-$-component of $\vt_1$ is zero) and integer part $\vx_0^I + \vg$, i.e., $$\vx_{(0,1)} + (1-\alpha_0) \vt_1= \vx'_1.$$
    However, $\vx_{(0,1)}$ has a zero in the $0_-$ and $0_+$-component, thus, as $\vx'_1$ has a positive $0_+$-component, $\vt_1$ also has that coordinate positive (again, we are excluding non-degenerated basic-integer solutions). Therefore, $\vx_1= \vx_{(0,1)}-\alpha \vt_1$ is unfeasible.

    We reproduce this procedure for the pairs $\vx_{(i-1,i)}, \vx_{i+1}'$ until $i=\ell$. This gives a sequence of degenerated points $\vx_{(i-1, i)}$ that can be constructed recurrently. In the final step, we have written $\bar \vx$ as a sum of $\vx_0$ and all the directions $\vt_i$. This way,
    $$\vt= \alpha_0 \vt_0 + ... + \alpha_{\ell} \vt_\ell$$
    where $\alpha_0+ ... + \alpha_\ell=1$.
    
\end{proof}
The conclusion of this lemma is that whenever we need to use a combination of directions of our double test set there is at least a degenerated point.

\begin{example}
    In example \ref{ex: lone_basis} we saw that we needed the direction $(-1,1,1)^T$ to keep improving. This direction is a combination of the basic-integer direction $\vt_1=(-2,0,1)^T$ or 
    $ \vt_2= (0,2,1)^T$ and the padded circuit $\vs=(1,1,0)^T$. We can see that $$\vt= \frac{1}{2} (-2,0,1)^T+ \frac{1}{2}(0,2,1)^T$$ and the degenerated point between both basis is $\vx_{(1,2)}=(0,0,3/2)^T$. \label{ex: lone_basis2}
\end{example}
\begin{remark}
    Lemma \ref{L: CombConv} tells us that the basis $B_i$ share $m-1$ columns with the basis $B_{i-1}$ and $B_{i+1}$, for any $i$. We say that a pair of bases sharing $m-1$ columns are adjacent (or neighbors). Each degenerated point is represented by using such a pair of basis. However, we note in passing that a degenerated point can be written in $n_R-(m-1)$ bases. Here, we are naming it using the subindices of only two of them, as these bases are the ones that provide us the basic-integer directions $\vt_i$ and $\vt_j$ that move us in and out of the degenerated point.
\end{remark}

These solutions can be characterized algebraically. Given two adjacent bases, $B_i$ and $B_j$, we can find a degenerated point in those two bases if and only if there exists $\vx^I \in \R^{n_I}$ such that:
$$ \left[ \begin{array}{c}
      \B_i^{-1} (\vb - A^I \vx^I) \\
       \zero
    \end{array} \right]
    = \left[ \begin{array}{c}
       \B_j^{-1} (\vb - A^I \vx^I) \\
       \zero
    \end{array} \right]$$
where $\B^{-1}_i$ and $\B^{-1}_j$ are $(m+1)\times m$ matrices formed using $B^{-1}_i$ and $B^{-1}_j$, respectively, adding a zero row corresponding to the column that does not share with the other basis.

We can rewrite this as a system on $\vb$ and $\vx^I$: 
\begin{equation}\tag{$\pi_{(i,j)}$} 
    \left[ \begin{array}{c | c}
     \B_i^{-1} - \B_j^{-1} & -\B_i^{-1} A^I + \B_j^{-1} A^I 
\end{array} \right]
\left[ \begin{array}{c}
     \vb \\
     \vx^I
\end{array} \right]  = \zero.\label{eq: Sistema H} 
\end{equation}

Let $H_{(i,j)}$ be the $(m+1)\times (m + n_I)$ matrix defining the previous system. It always has the zero solution. For non-trivial solutions to exist, the rank of $H_{(i,j)}$ must not be maximum. We can check that the rank of the matrix is the rank of $\B_i^{-1}-\B_j^{-1}$, which is at most $m$. Now, we are interested on the solutions that have $\vb$ integer and $\vx^I$ positive. We may assume that $\vx^I \not \in \Z^{n_I}$. These solutions define a set that we call $\Pi_{(i,j)}$. 

Let us get back to the set up described at the beginning of the section. In order to avoid confusion, we consider the right hand side vector $\bar \vb$ to be fixed, transforming $\pi_{i,j}$ into
\begin{equation}
\tag{$\pi_{(i,j)}(\bar \vb)$}
    (\B^{-1}_i - \B^{-1}_j)A^I \vx^I = 
    (\B^{-1}_i -\B^{-1}_j) \bar \vb, \label{eq: Sistema pi b}
\end{equation}
where we call $\Pi_{(i,j)}(\vb)$ to the set of solutions with $\vx^I \in \R^{n_I}_{\geq 0}$.

\begin{example}
    Let us take a look back at example \ref{ex:unpointed}:
    $$\{\min \vc^T \vx ~|~ x_1 - x_2 + 2 x_3 + 2x_4 = 1, x_1, x_2 \in \R_+, x_3, x_4 \in \Z_+\}.$$ 
    The pair $(\T^{A,b}, \C(A^R))$ where 
    $$\T^{A,b}= \big \{ \pm(-2, 0, 1, 0)^T, \pm(-2, 0, 0, 1)^T,
       \pm (0, 2, 1, 0)^T, \pm(0, 2, 0, 1)^T, \pm(0, 0, 1, -1)^T\big \}$$
    and $C(A^R)= \{(1,1,0,0)^T\}$ is a double test set.
    Now, in order to move from $\vx_0=(1,0,0,0)^T$ to $\vx_{1,0}= (0,1,1,0)^T$ or $\vx_{0,1}=(0,1,0,1)^T$ we need the double augmentation directions from the double test set $$\vx_{1,0}- \vx_0= (-1,1,1,0)^T= (-2,0,1,0)^T+ (1,1,0,0)^T$$ or
    $$\vx_{0,1}- \vx_0 =(-1,1,0,1) ^T= (-2,0,0,1)^T+ (1,1,0,0)^T.$$
    Lemma \ref{L: CombConv} tells us that there are degenerated points, we calculate them. In this particular example, $\bar b=1$, $B_0=1$ and $B_1=-1$, hence 
    $$\B_0^{-1} = \left( \begin{array}{c}
        1  \\
        0  
    \end{array} \right) \quad \text {and} \quad 
    \B_1^{-1} = \left( \begin{array}{c}
        0  \\
        -1  
    \end{array} \right). $$
    This way, our matrix $H_{(0,1)}$ is
    $$\left[ \begin{array}{c | c}
     \B_0^{-1} - \B_1^{-1} & -\B_0^{-1} A^I + \B_1^{-1} A^I 
    \end{array} \right] = 
    \left( \begin{array}{c | c c}
         1 &-2 &-2 \\
         1 &  -2 &  -2
    \end{array} \right)$$
    providing the system $\pi_{(0,1)}(1)$
    $$
    \left( \begin{array}{cc}
         2& 2 \\
          2& 2
    \end{array} \right)
    \vx^I = \left(
    \begin{array}{c}
         1\\
         1
    \end{array} \right)$$
    which yields the linear subspace given by $$\Pi_{(0,1)}(1)=\{ (x_3, x_4) \in \R^2_{\geq 0} \, \, | \, \, x_3 + x_3 = 1/2\}.$$ \label{ex:unpointed2}
\end{example}

Now that we know how to calculate degenerated points, we shall use them to obtain improving directions. The idea is simple, we replicate the proof of Lemma \ref{L: CombConv}: Start with a basic-integer solution and search for all possible degenerate points that are ``nearby''. First, we define this idea of proximity:

 The reader may observe that $\alpha_0$ is actually dependent on $(\vx_0^I, \vx^I, \vg)$, although for the sake of simplicity, will be referred simply as $\alpha_0$.

\begin{definition}
    Let $\vx_0= (\vx_0^R, \vx_0^I)^T$ be any basic-integer solution and $\vg$ an integer part of $\phi(\gmip)$. Consider any other point $\vx= (\vx^R, \vx^I)^T$ such that $\vx_I = \alpha_0 \vg + \vx_0$ for $\alpha_0 \in [0,1]$. We call $\Adj(\vx, \vx_0, \vg)$ to the set of all degenerated solutions $\bar \vx= (\bar \vx^R, \bar \vx^I)^T$ such that the bases of $\bar \vx$ and $\vx$ are adjacent (or neighbor) and $\bar \vx^I = (\bar \alpha + \alpha_0) \vg + \vx_0^I$ with $\bar \alpha + \alpha_0 \in [0,1]$.\label{Def: Adj1}
\end{definition}

\begin{remark}
    Under this circumstances, there are three possible scenarios for $\vx$ depending on $\alpha_0$:
    \begin{enumerate}
        \item[·]$\alpha_0 =0$, hence $\vx=\vx_0$,
        \item[·]$\alpha_0=1$, therefore, $\vx$ is any basic-integer solution such that $\vx^I = \vx_0 + \vg$,
        \item[·] $\alpha_0 \in (0,1)$  and $\vx$ is any degenerated point that is close to $\vx_0$, i.e.,  $\vx= \alpha_0\vg + \vx_0^I.$
    \end{enumerate}
   In order to keep track of this distance, we define $\alpha( \vx^I, \vx_0, \vg)$ as $\alpha_0$.
\end{remark}

\begin{lemma}
  The set $\Adj(\vx^I, \vx_0^I, \vg)$ is either empty or finite. \label{L: Adj_es_Finito}
\end{lemma}
\begin{proof}
    Let $\alpha_0= \alpha(\vx^I, \vx_0^I, \vg)$. If $\alpha_0=1$, the set $\Adj(\vx, \vx_0, \vg)$ is empty. We consider $\alpha_0=0$ as if $\alpha_0\in (0,1)$, the following reasoning applies for each basis of $\vx^R$.
    
    Let $B_i:=B_0$ be the basis of $\vx_0$. In this set up, it is sufficient to show that for any adjacent basis, $B_j$, to the basis of $\vx_0$, there are a finite number of solutions $\bar \vx^I$ to \eqref{eq: Sistema pi b}, such that $\bar \vx^I$ can be written as $\bar \vx^I= \vx_0^I + \alpha \vg$, for $\alpha \in [0,1]$. 

    We assume that such a $\bar \vx^I$ exists. By definition, it is the intersection of the line $\vx_0^I + s \vg$ with the vectorial subspace of $\R^{n_I}$ given \eqref{eq: Sistema pi b}.
     
    Either this intersection is unique or it is the hole line. In order for the second scenario to happen, $(\B_0^{-1}-\B_j^{-1})A^I \vg $ must be $\zero$. Then, using the basic-integer direction of integer part $\vg$ and basis $B_0$ we can check that $\vx_0$ is also a degenerated solution. However, we assume at the beginning of the section that all basic-integer solutions would be non-degenerated.
\end{proof}

This way, when we reach a new degenerated point, we can either search for another one, if any , or go to a basic-integer solution using a basis of our last degenerated point. In order to do that, we need to ``complete'' the direction we have been using:
\begin{definition}
    Let $\vg$ be the integer part of a direction, and let 
    $\vt= (\vq_0, \vg_0)^T$ be another direction such that $\vg_0= \alpha \vg$ for $\alpha \in [0,1]$. Now, consider a basic-integer solution $\vx_0$. We denote $\Compl(\vt, \vg, \vx_0)$ as the set of directions $\vt +(1-\alpha)\vt_j$ where $\vt_j$ is any basic-integer direction using a basis of $\vx_0+ \vt$ and $\vg$. \label{Def: Comp1}
\end{definition}

\bigskip

In the following, we describe the procedure, Algorithm \ref{alg: sol_deg}, that given a non-optimal solution returns an improving direction. 

\begin{algorithm}[H]
    \caption{Completion Procedure: using degenerated points}\label{alg: sol_deg}
    \begin{algorithmic}[1]
        \State \textbf{input} $\vx_0=(\vx^R_0, \vx^I_0)^T$ a basic-integer feasible non optimal solution of \eqref{eq:standard} and the integer part $\vg$ of an improvement direction $\vt =(\vq, \vg)^T$ of $\vx_0$.
        \State \textbf{output} $T_0$, a set containing $(\vq, \vg)^T$.
        
        \State $C= \Adj(\vx_0, \vg )$
        \State $T_0 = \emptyset$
        \While{$C \neq \emptyset$}
            \State $(\vx^R, \vx^I)^T= \vx \gets $ an element in $C$
            \State $C \gets C \setminus \{\vx \}$
            \State $\alpha_0 = \alpha(\vx_0^I, \vx^I, \vg)$
            \If{$\alpha_0 \neq 1$}:
                \State $\vt = \vx - \vx_0$
                \State $T_0 \gets T_0 \cup \Compl(\vt, \alpha_0, \vx)$
                \State $C \gets C \cup \Adj(\vx, \vx_0, \vg)$
            \EndIf
        \EndWhile
    \end{algorithmic}
\end{algorithm} 

Algorithm \ref{alg: sol_deg} terminates as the number of degenerated points is finite. Moreover, the output of the algorithm, $T_0$, satisfies that it contains $\vt=(\vq,\vg)^T$ due to Lemma \ref{L: CombConv}.

The following example illustrates the application of Algorithm \ref{alg: sol_deg}.
\begin{example}
     [Continuation of Example \ref{ex:unpointed2}] We have already computed, $$\Pi_{(0,1)}(1)=\{ (\vx_3, \vx_4) \in \R^2_{\geq 0} \, \, | \, \, x_3 + x_4 = 1/2\},
    $$ and we know that
    $(1,0)^T, (0,1)^T \in \phi(\gmip(A))$.

    We start with $\vg_0=(1,0)^T$. In order to move out of $\vx_0=(1,0,0,0)^T$ using $\vt_0=(2,0,1,0)^T$, we must search for an $\alpha_0$ such that $\vx_0 + \alpha_0 \vg_0= (\alpha_0,0)$ belongs to $\Pi_{(0,1)}(1)$. This way
    $$\alpha_0+0=1/2.$$
    In conclusion, the direction we are looking for is given by $\vt= \frac{1}{2}\vt_0+ \frac{1}{2}\vt_1$ where $\vt_0=(0,-2,1,0)^T$ is the basic-integer direction of basis $B_1$ and integer part $\vg_0$. For $\vg_2$ the result is analogue.
\end{example}
\subsection{Universality}
If we perform a detailed analysis of Algorithm \ref{alg: sol_deg} and Lemma \ref{L: CombConv}, we observe that we can reformulate the algorithm starting with a degenerated point $\vx= (\vx^R, \vx^I)$. In order to obtain a direction using $\vx$, assuming that there are not any more degenerated points in between, there must be an integer part $\vg \in \phi(\gmip(A))$ such that there exist two solutions $\vx_i, \vx_j$, verifying $\vx_i^I = \vx + \alpha \vg, \vx_j^I = \vx -(1- \alpha) \vg,$ with  $\alpha \in [0,1]$. In this scenario, the direction would be $\alpha \vt_i + (1-\alpha) \vt_j$ where $\vt_i$ and $\vt_j$ are basic-integer directions of integer part $\vg$ and the basis of $\vx_i$ and $\vx_j$. Here, the right hand side vector does not appear explicitly, but it is necessary to calculate $\alpha$. However, for fixed $\vg$, there are a finite number of possible values  for the parameter $\alpha$.

\begin{example}
    Consider the problem $$\{ \vc^T \vx ~|~ x_1 - x_2 + 2 x_3 = b, x_1,~ x_2 \in \R_+, ~ x_3 \in \Z_+\},$$
    which is a generalization of Example \ref{ex: lone_basis}. Here there are only two bases, $B_1=1$ and $B_2=-1$, and $\Pi_{(1,2)}$ can be easily calculated: 
    $$\Pi_{(1,2)}= \{ (b, \frac{b}{2}) \, \, | \, \,   b \in \Z_{\geq 0} \}.$$
    As we have already seen, $\G^{A}= \{ \pm 1 \}$, so we consider $\vg=1$. Then, in order to build the direction $\vt= \alpha \vt_1 + \beta \vt_2$ where $\vt_1$ and $\vt_2$ are the basic-integer directions of integer part $\vg$ and basis $B_1$ and $B_2$, respectively, we need to find an $\alpha \in [0,1]$ such that $\alpha \vg + \vx^I_{(1,2)} \in \Z^{n_I}$. In this case, $\vg=1$ and $\vx^I_{(1,2)} = b/2$, so $\alpha = 1/2$ (if $b$ is an even number, the integer part of the degenerated point is actually integer) and the direction $\vt$ is $(-1,1,1)^T$, as specified earlier.

    In conclusion, for this example, we have proven that the set $\T^{A} \cup \{(-1,1,1)^T \} \cup \C(A^R)$ is a finite universal test set for $A$ since the direction $(-1,1,1)^T$ is completely independent of the choice of $\vb$.
\end{example}

This example shows that, although $\vb$ and the integer part of the degenerated points are linearly related, the coefficient $\alpha$ that gives the improving direction only depends on the fractional component of the ``integer'' part of the degenerated point. 

\begin{definition}
    Given $\vg \in \phi (\gmip(A))$, we call $I_\vg$ to the set of all the coefficients $\alpha \in [0,1]$ verifying that there exist $\vx^I \in \R^{n_I}, \vb \in \Z^m$ and subindices $i,j$ such that $(\vb, \vx^I) \in \Pi_{(i,j)}$ satisfying $\vx^I- \alpha \vg \in \Z^{n_I}$.
\end{definition}

\begin{lemma}
    The set  $I_\vg$ is finite. \label{L: Ig_finite}
\end{lemma}
\begin{proof}
    We show that, for a given pair of bases $B_i$ and $B_j$, the set formed by all $\alpha \in [0,1]$ such that $\vx^I- \alpha \vg \in \Z^{n_I}$ for any solution $(\vb, \vx^I) \in \Pi_{(i,j)}$ is finite. 
    Consider any solution of $(\vb, \vx^I)$ in $\Pi_{(i,j)}$ such that $\vx_0^I=\vx^I -\alpha  \vg \in \Z^{n_I}$ for some $\alpha \in [0,1]$. We know that $(\vb,\vx^I)$ satisfies 
    $$(\B^{-1}_i -\B^{-1}_j)  \vb = (\B^{-1}_i - \B^{-1}_j)A^I \vx^I,$$
    so 
    $$(\B^{-1}_i -\B^{-1}_j)  \vb = (\B^{-1}_i - \B^{-1}_j)A^I (\vx_0^I + \alpha \vg),$$
    and we end up with the following expression:
    $$(\B^{-1}_i -\B^{-1}_j)(\vb-A^I \vx_0^I) =\alpha  (\B^{-1}_i - \B^{-1}_j)A^I \vg.$$
    Now, $(\B^{-1}_i - \B^{-1}_j)A^I \vg$ is a vector of $m+1$ coordinates. Let $k$ be the index of any of these coordinates that is not equal to zero. If $(\B^{-1}_i - \B^{-1}_j)A^I \vg= \zero$ then, as in the proof of Lemma \ref{L: Adj_es_Finito}, $\vx$ is a degenerated solution, by hypothesis, this scenario is not possible. With this in mind,
    \begin{equation}
        \alpha = \frac{[(\B^{-1}_i - \B^{-1}_j)(\vb-A^I\vx_0^I)]_k}{[(\B^{-1}_i - \B^{-1}_j)A^I \vg]_k}. \label{eq: alfa_explicito}
    \end{equation}
    In the expression above, both $\vb$ and $\vx_0^I$ have integer components. Therefore, we can rationalize the rest of coefficients appearing in the numerator and denominator and $\alpha$ results in a quotient of integer numbers.  Let $D_k$ be the integer on the denominator and $f_k$ the linear transformation of integer coefficients that, given $\vb$ and $\vx_0^I$, returns an integer that is the numerator, therefore,
    \begin{equation}
       \alpha = \frac{f_k(\vb, \vx_0^I)}{D_k}. \label{eq: alfa_simplificado}
    \end{equation}
    Note that $f_k$ and $D_k$ depend on the bases $B_i$ and $B_j$, however, for the sake of simplicity and whenever it does not induce any confusion, we do not write these dependency explicitly.
    
    To sum up, any $\alpha \in I_{\vg}$ can be written as the ratio of a certain linear transformation, $f_k$, evaluated at a particular point $(\vb, \vx_0^I)$ and a certain integer number $D_k$. As $\alpha \in [0,1]$, there are at most $D_k +1$ possible $\alpha$ values. Those values can be written as $\ell$ divided by $D_k$ with $\ell \in \{0, ..., D_k\}$.
\end{proof}

\begin{remark}
    In order to obtain $I_\vg$ we need to tune a little bit the previous proof: $\alpha \in I_\vg$ if and only if there exists a pair of bases $B_i, B_j$ such that for all $k$ verifying $[(\B_i^{-1}- \B_j^{-1})A^I \vg]_k \neq 0$, if we calculate $D_k$ we can write $\alpha$ as $ \ell_k/ D_k$ for certain $\ell_k \in \{0, ..., D_k$\}.
\end{remark}

\begin{remark}
    The set $I_g$ is partially-ordered. Note that each element $\alpha $ of $I_\vg$ has associated (at least) a pair of bases. We write $\alpha_{i,j}$ where $i,j$ are the subindices of these bases. 
\end{remark}

We want to replicate Algorithm \ref{alg: sol_deg}. We start at $\alpha=0$, and we choose a path of $\alpha_{i,j} \in I_\vg$ such that each $\alpha_{i,j}$ is adjacent to the previous one. This way, starting with $\vt= \zero$, every time we move from some $\alpha_{i,j}$ to $\alpha_{i',j'}$ we add $(\alpha_{i',j'}-\alpha_{i,j})\vt_k$ to $\vt$, where $\vt_k$ is the basic-integer direction of integer part $\vg$ and basis $B_k$, the basis that both $\alpha$ have in common.

We need to adjust the definitions \ref{Def: Adj1} and \ref{Def: Comp1} to this set-up:

\begin{definition}
    Given an integer part $\vg$, two bases $B_i, B_j$ and a coefficient $\alpha_{i,j} \in I_\vg$, let  $\Adj(\alpha_{i,j}, \vg)$ be  the set  of all $\alpha_{i', j'} \in I_{\vg}$ where $\alpha_{i',j'}>\alpha_{i,j}$ and one of the two bases defining $\alpha_{i', j'}$ is adjacent to one of the bases defining $\alpha_{i,j}$. If $\alpha_{i,j}=0$, we consider that it can be written in any basis, so all $\alpha' \in I_\vg$ are adjacent.\label{Def: Adj2}
\end{definition}

In addition, in a similar way as we have done with $\alpha=0$, we add $\{1\}$ to this set, as the final augmentation has $\alpha=1$ to be the end point of all paths and we consider that it can be written in any basis. 

\begin{definition}
    Given an integer part $\vg$, a direction $\vt$, a pair of bases $B_i, B_j$ and a coefficient $\alpha_{i,j} \in I_\vg$,  $\Compl(\vt, \alpha_{i,j}, \vg)$ is  the set of all pairs $$(\alpha_{i', j'}, \vt + (\alpha_{i',j'}- \alpha_{i,j}) \vt_k)$$ where $\alpha_{i',j'} \in \Adj(\alpha_{i,j}, \vg)$ and $\vt_k$ is the basic-integer direction of integer part $\vg$ and the common basis shared in the definition of $\alpha_{i,j}$ and $\alpha_{i',j'}$.
\end{definition}

\begin{algorithm}[ht]
    \caption{Completion Procedure: Using fractional parts}\label{alg: sol_frac}
    \begin{algorithmic}[1]
        \State \textbf{input} $\vg \in \phi(\gmip(A))$ and  $I_\vg$ .
        \State \textbf{output} $T_\vg$, a finite set containing directions of integer part $\vg$.
        
        \State $T_\vg= \emptyset$
        \State $C = \{(0, \zero)\}$
        \While {$C\neq \emptyset $}
            \State $(\alpha_{i,j}, \vt) \gets $ an element in $C$
            \State $C \setminus \{(\alpha_{i,j}, \vt)$ \}
            \If{$\alpha_{i,j}=1$}
                \State $T_\vg = T_\vg \cup \{\vt\}$
            \Else
            \State $C \gets$ $C \cup \Compl(\vt, \alpha_{i,j}, \vg)$
            \EndIf
        \EndWhile
    \end{algorithmic}
\end{algorithm}
The algorithm \ref{alg: sol_frac} ends since the set $I_\vg$ is finite and whenever we have an $\alpha$ the adjacent ones are strictly bigger.

In order to proceed further, we denote by
$T_\vg$ the output of Algorithm $\ref{alg: sol_frac}$.

\begin{definition}
    The set $\T(A^R, A^I)$ (or simply $\T(A)$) is the union of all sets $T_\vg$ for $\vg \in \phi(\gmip(A))$ and the set of padded circuits $\C(A^R)$.
\end{definition}

\begin{theorem}
    The set $\T(A)$ is a \concept{finite universal test set} for any MIP program with matrix $A=(A^R | A^I)$.
\end{theorem}
\begin{proof}
    In order to prove this theorem we have to show that given a specific \ref{eq:standard} program with matrix $A$ and a solution $\vx=(\vx^R, \vx^I)^T$, if $\vx$ is not an optimal solution, we are able to find a direction $\vt \in \T(A)$ such that $\vx+\vt$ is feasible and $\vc^T \vt < 0$.

    First of all, without loss of generality, we  can assume that $\vx$ is a non-optimal basic-integer solution. Otherwise, we can use the padded circuits to find an improvement that is basic-integer, as we know that there is an optimal solutions of $\eqref{eq:slice}$ which is basic-integer. 
    Let  $B_0$ be the basis that defines $\vx$ (in case of having more than one basis, i.e, being degenerated, we check for all of them). As $\gmip(A)$ is a universal test set, there exists a direction $\vt \in \gmip(A)$ such that $\vx + \vt$ is feasible and an improvement over $\vx$. Again, without loss of generality, we suppose $\vx + \vt$ to be a basic-integer direction (if it is not, we add circuits until it is). We must show that $\vt \in \T(A)$.

    Let $\vt$ be the direction $(\vq, \vg)^T$. By Lemma \ref{L: CombConv} this direction can be written as a convex combination
    $$\vt= \alpha_0 \vt_0+ ... + \alpha_\ell \vt_\ell$$
    of some basic-integer directions of integer part $\vg$ and  bases adjacent to the previous one, being  $B_0$ the basis of $\vt_0$. This way, we know that there are $\ell-1$ degenerated points, namely
    $$\vx_{(k-1,k)}= \vx + \sum_{i=0}^k \alpha_i \vt_i \quad \text{ for } k \in \{1,...,\ell-1 \},$$
    and each $\hat \alpha_k = \sum_{i=0}^k \alpha_i$ for $k \in \{1,...,\ell-1 \}$ belongs to $I_\vg$. Finally, we apply Algorithm \ref{alg: sol_frac} running over $I_\vg$  starting with $\hat \alpha_0= \alpha_0$. The algorithm evaluates all possible combinations including the one using the scalars $\hat \alpha_0, ..., \hat \alpha_\ell$, that generates the direction $\vt$. This proves that $\vt \in T_\vg$, and therefore $\vt \in \T(A)$.
\end{proof}
We finish this section providing a universal test set for the matrix $A=[1,-1 |2, 2]$ of example \ref{ex:unpointed}.
\begin{example}
    Given the matrix $A=[1,-1 |\, 2, 2]$ we know that 
   $$\phi(\gmip) \subset \G^*= \big \{\pm(1,0)^T, \pm(0,1)^T, \pm(1,-1)^T \big \}.$$
    First, we calculate $I_\vg$ for any $\vg \in \phi(\gmip)$ using the proof of lemma \ref{L: Ig_finite}. It is particularly easy to compute as there are only a couple of basis, $B_0=1$ and $B_1$. Thanks to \ref{ex:unpointed2} we have:
    $$H_{(0,1)}= (\B_0^{-1}-\B_1^{-1} | -(\B_0^{-1}-\B_1^{-1})A^I) )= \left( \begin{array}{c | c c}
         1 &-2 &-2 \\
         1 &  -2 &  -2
    \end{array} \right).$$
    Now, we consider $\vg_0= (0,1)^T$ and we multiply $-(\B_0^{-1}-\B_1^{-1})A^I)$ by $\vg_0$:
    \[\left( \begin{array}{c c}
         -2&-2 \\
         -2&-2 
    \end{array} \right) 
    \left( \begin{array}{c}
         0 \\
          1
    \end{array} \right)
    = \left( \begin{array}{c}
         -2 \\
         -2  
    \end{array} \right). \]
    We can see that both components of the vector are non zero. At this point we have to write the numerator 
    $$(\B_0^{-1}-\B_1^{-1})(\vb - A^I \vx_0^I)$$ as a vector of rational components independent of the choice of $b$ and $\vx_0^I$:
    \[ \left( \begin{array}{c}
         1  \\
         1 
    \end{array} \right) \left[
    b
    - \left(1,1 \right)    
    \left( \begin{array}{c}
         (\vx_0^I)_1  \\
         (\vx_0^I)_2
    \end{array} \right) \right] =
    \left( \begin{array}{c}
          b-(\vx_0^I)_1-(\vx_0^I)_2 \\
          b-(\vx_0^I)_1-(\vx_0^I)_2
    \end{array} \right).
    \]
    Therefore, expression \eqref{eq: alfa_explicito} and \eqref{eq: alfa_simplificado} are
    $$\alpha = \frac{b-(\vx^I_0)_1- (\vx_0^I)_2}{2}$$
    and $I_{\vg_0}= \{0, \frac{1}{2}, 1 \}$. Running the algorithm \eqref{alg: sol_frac} over this set we get the set
    $$T_{\vg_0}= \{(-2,0,0,1)^T, (0,2,0,1) ^T, (-1,1,0,1) ^T \}.$$

    If we repeat this calculations for the rest of the elements in $\G^*$, we get that a universal test set for the matrix $A$ is
    \begin{equation*}
    \T = \left\{
    \begin{array}{c}
       \pm (-2,0,0,1)^T, \; \pm (0,2,0,1)^T, \; \pm (-1,1,0,1)^T, \\
       \pm (-1,1,1,0)^T, \; \pm (0,2,1,0)^T, \; \pm (-2,0,1,0)^T, \\
       \pm (0,0,1,-1)^T
    \end{array}
    \right\}.
    \end{equation*}
\end{example}

The interested reader is referred to \href{https://github.com/01Jara/MIPS}{https://github.com/01Jara/MIPS} where an implementation of the algorithms in the paper can be found. The authors have used these codes to solve the examples in the paper. Although the methodology is valid and correct, its bottleneck is the computation of the Hilbert bases and the huge size of the resulting test sets. To illustrate the complexity of this procedure, in some of our computational experiments for MIP problems with 2 constraints and 6 variables, two integer, the size of the set $\T$ was greater than 80.000. This makes our approach not effective for solving actual MIP, but it provides a certificate of the existence of finite, universal test sets for general MIP.

\section{ Concluding remarks} \label{se:conclusions}

This paper introduces the concept of double test set and proves the existence of  such a set for MIP. Although, it does not directly  provides a standard finite test set, from combination of elements from the two sets, one can further exploit that structure to prove and algorithmically construct  a standard finite test set for MIP. These sets are both used to develop augmentation algorithms for MIP.

Many questions have emerged from our algebraic analysis of MIP. For some of them, we have provided positive or negative answers. However, even in the positive case, many challenges still remain.
We mention two of them. Our positive answer to the existence of computable test sets for MIP brings some interesting new challenges, one of them is the identification of a suitable (finite) augmentation scheme able to be competitive in solving mixed integer problems. In addition, it may be interesting to use mixed approaches combining augmentation techniques with some other solution strategies for MIPs, may be alternating different steps in the algorithms. As far as we are concerned, this has not been explored and may be the subject of a follow up paper.

\section*{Declarations}

This research has been partially supported by the grants PID2020-114594GB-C21 funded by MICIU/AEI /10.13039/ 501100011033, PCI2024-155024-2 - ``Optimization over Nonlinear Model Spaces: Where Discrete Meets Continuous Optimization'' funded by Agencia Estatal de Investigación and EU funds; and IMUS-Maria de Maeztu grant CEX2024-001517-M - Apoyo a Unidades de Excelencia María de Maeztu.

\noindent

\bibliography{sn-bibliography.bib}

\end{document}